\documentclass[a4paper]{amsart}

\title{Realizing uniformly recurrent subgroups}
\date{\today}

\author{Nicol\'{a}s Matte Bon}
\address{
  D-MATH -- ETH Z\"urich \\
  R\"amistrasse 101 \\
  8092 Z\"urich \\
  Switzerland}
\email{nicolas.matte@math.ethz.ch}

\author{Todor Tsankov}
\address{
  Institut de Math{\'e}matiques de Jussieu--PRG \\
  Universit\'e Paris Diderot \\
  75205 Paris \textsc{cedex} 13 \\
  France
  -- and --
  D{\'e}partement de math{\'e}matiques et applications \\
  {\'E}cole normale sup{\'e}rieure \\
  75005 Paris \\
  France}
\email{todor@math.univ-paris-diderot.fr}

\usepackage{hyperref}
\usepackage[initials,shortalphabetic]{amsrefs}
\usepackage[utf8]{inputenc}

\usepackage[T1]{fontenc}
\usepackage[osf]{mathpazo}
\usepackage{amssymb}
\usepackage{eucal} 

\usepackage{my-macros}
\usepackage{enumitem}
\setlist[enumerate,1]{label=(\roman*), font=\normalfont}

\newcommand{\Sc}{\mathcal{S}}

\newcommand{\Sub}{\operatorname{Sub}}
\newcommand{\stab}{\operatorname{Stab}}
\newcommand{\Cub}{\operatorname{C_{ub}}}

\DeclareMathOperator{\Fix}{Fix}
\DeclareMathOperator{\Mov}{Mov}

\begin{document}

\maketitle

\begin{abstract}
We show that every uniformly recurrent subgroup of a locally compact group is the family of stabilizers of a minimal action on a compact space. More generally, every closed invariant subset of the Chabauty space is the family of stabilizers of an action on a compact space on which the stabilizer map is continuous everywhere. This answers a question of Glasner and Weiss. We also introduce the notion of a universal minimal flow relative to a uniformly recurrent subgroup and prove its existence and uniqueness.
\end{abstract}


\section{Introduction}
Let $G$ be a locally compact group. Consider the space of subgroups $\Sub(G)$ endowed with the Chabauty topology \cite{Chabauty1950} and recall that a subbasis of open sets for this topology is given by sets of the form 
\[\mathcal{U}_C=\{H\in \Sub(G) : H\cap C=\varnothing\} \quad\text{ and } \quad \mathcal{U}_V=\{H\in\Sub(G) : H\cap V\neq \varnothing\},\]
where $C$ varies among the compact subsets and $V$ among the open subsets of $G$. 
This topology makes $\Sub(G)$ a compact Hausdorff space on which $G$ acts continuously by conjugation.


Glasner and Weiss~\cite{Glasner2015} initiated the study of \emph{uniformly recurrent subgroups} (URS for short), i.e., closed, invariant, minimal subsets of $\Sub(G)$. This notion can be seen as a topological analogue of the measure-theoretic one of \emph{invariant random subgroup}, a terminology coined in \cite{Abert2014}. URSs have recently attracted some attention as it turned out that they are a convenient tool to study \emph{boundary actions} which, for discrete groups, are connected to $C^*$-\emph{simplicity} \cite{Kennedy2015p, LeBoudec2016p}. 
 
As was shown by Glasner and Weiss \cite{Glasner2015}, a URS is naturally associated to every minimal action $G \actson X$ on a compact space. Namely, consider the stabilizer map $\stab \colon X \to \Sub(G)$. This map is usually not continuous. However, it is upper semi-continuous in the sense that for every net $(x_i)$ converging to $x\in X$, every cluster point of $\stab(x_i)$ in $\Sub(G)$ is contained in $\stab(x)$. This property is enough to ensure that the closure of the image of $\stab$ contains a unique URS. (This result is proved in \cite[Proposition 1.2]{Glasner2015}. See also the argument of \cite[Lemma I.1]{Auslander1977} to avoid the assumption, made throughout \cite{Glasner2015}, that $X$ is metrizable.) The unique URS contained in  $\overline{\stab(X)}$ is called  the \emph{stabilizer URS} of $G \actson X$ and will be denoted by $\Sc_G(X)$.

Conversely, Glasner and Weiss ask whether every URS arises as the stabilizer URS of a minimal action. This question is motivated in part by an analogous result for invariant random subgroups, established in \cite{Abert2014} for countable groups and in \cite{Abert2012p} for general locally compact groups. However, the method of proof of these results does not translate easily to this topological dynamical setting.

 In this paper, we answer the question of Glasner and Weiss question in the affirmative. More generally, we show the following.
\begin{theorem}
  \label{t:Chabauty}
  Let $G$ be a locally compact group and let $\mathcal{H} \sub \Sub(G)$ be a closed, invariant subset. Then there exists a continuous action of $G$ on a compact space $X$ such that the stabilizer map $\stab \colon X \to \Sub(G)$ is everywhere continuous and its image is equal to $\mathcal{H}$. If $G$ is second countable, $X$ can be chosen to be metrizable.
\end{theorem}

In the above result, $\cH$ is not assumed to be a URS, and therefore the action $G \actson X$ cannot be chosen to be minimal in general. However if $\cH$ is a URS, the continuity of $\stab$ implies that every minimal invariant subset $Y$ of $X$ verifies the same conclusion.
  \begin{cor}\label{c:URS:stab}
 Every URS of a locally compact group arises as the stabilizer URS of a minimal action on a compact space. Moreover, the action can be chosen so that the stabilizer map is continuous.
 \end{cor}
If $\cH=\set[big]{\set{1_G}}$, Theorem~\ref{t:Chabauty} recovers a classical theorem in topological dynamics, due to Veech~\cite{Veech1977} (and previously to Ellis~\cite{Ellis1960} for discrete groups), stating that every locally compact group admits a free action on a compact space. The proof of Theorem~\ref{t:Chabauty} is inspired by the proof of this result.

 \medskip 
 
Recall that for every topological group $G$ there exists a minimal compact $G$-space $M(G)$ which is \emph{universal}, meaning that every other minimal compact $G$-space is an equivariant continuous factor of $M(G)$. While the existence of $M(G)$ is not difficult to establish, its uniqueness up to conjugation is more delicate and was proved by Ellis \cite{Ellis1960} using his theory of right topological semigroups. An equivalent formulation of Veech's theorem is therefore that if $G$ is locally compact, then the action of $G$ on $M(G)$ is free. 

We show that among spaces satisfying Corollary~\ref{c:URS:stab} there is a unique {universal} one in the following sense. Given a locally compact group $G$ and a URS $\cH$ of $G$, we say that a compact $G$-space $X$ is \emph{subordinate} to $\cH$ if every subgroup in $\cH$ fixes a point of $X$. 

\begin{theorem}
Let $G$ be a locally compact group, and $\cH$ be a URS of $G$. Then there exists a unique, compact, minimal $G$-space $M(G, \cH)$ which is subordinate to $\cH$ and is universal among minimal compact $G$-spaces that are subordinate to $\cH$.
\end{theorem}

We call $M(G, \cH)$ the \emph{universal minimal flow of $G$ relative to $\cH$}. Corollary~\ref{c:URS:stab} is thus equivalent to saying that the stabilizer map $\stab\colon M(G, \cH)\to \Sub(G)$ is continuous and its image is precisely $\cH$.

Finally, we characterize under what conditions the space $M(G, \cH)$ is metrizable (see Theorem \ref{t:metrisable}).


\subsection*{Related work} In an independent work \cite{Elek2017p} that appeared while this paper was being completed, G.~Elek proves Corollary~\ref{c:URS:stab} for finitely generated groups using a  different method. In another recent preprint, T.~Kawabe~\cite{Kawabe2017p} has obtained a proof of Corollary~\ref{c:URS:stab} for countable discrete groups when the URS consists of amenable subgroups.

\subsection*{Acknowledgements} We are grateful to Adrien Le Boudec for useful discussions. We would also like to thank Uri Bader and Pierre-Emmanuel Caprace for indicating that \cite[Lemma I.1]{Auslander1977} allows to avoid the metrizability assumptions in \cite{Glasner2015} in the definition of a stabilizer URS. Finally, we are grateful to Eli Glasner and to the anonymous referee for useful remarks and suggestions. Research on this paper was partially supported by the ANR projects GAMME (ANR-14-CE25-0004) and AGRUME (ANR-17-CE40-0026).

\section{Proof of Theorem \ref{t:Chabauty}}
\subsection{Case of discrete groups}
If $G$ is a discrete group, Theorem~\ref{t:Chabauty} is considerably simpler, therefore we begin by giving a proof in this special case.  

Let $Z$ be a discrete set endowed with a group action $G \actson Z$ and $g \in G$. We denote by $\Fix_g(Z)$ the set of points fixed by $g$, and by $\Mov_g(Z)$ its complement. 
 Let $\beta Z$ be the Stone--\v{C}ech compactification of $Z$. By the universal property of the Stone--\v{C}ech compactification, the action of  $G$ on $Z$ extends to an action by homeomorphisms on $\beta Z$.   Given a subset $W \sub Z$, the notation $\overline{W}$ refers to the closure in $\beta Z$.
 
Ellis has shown in \cite{Ellis1960} that the action $G \actson \beta G$ is free. The following lemma is essentially a generalization of this fact.

\begin{lemma}
  \label{l:Stone-Cech:discrete}
Let $G\actson Z$ be a group action on a discrete set $Z$. Then for every  $g\in G$ we have  $\Mov_g(\beta Z)= \cl{\Mov_g(Z)}$. In particular, the stabilizer map $\stab \colon \beta Z \to \Sub(G)$ is continuous.
\end{lemma}
\begin{proof}
  Clearly $\Fix_g(\beta Z) \supseteq \cl{\Fix_g(Z)}$. Moreover, $\beta Z = \Fix_g(\beta Z) \sqcup \Mov_g(\beta Z) = \cl{\Fix_g(Z)} \cup \cl{\Mov_g(Z)}$ (the second equality follows from the density of $Z$), and therefore  $\Mov_g(\beta Z) \sub \cl{\Mov_g(Z)}$. Let us check the reverse inclusion. We can find a function $f \colon Z \to \{0,1,2\}$ with the property that for every $x \in \Mov_g(Z)$, we have $|f(gx)-f(x)| \geq 1$ (such a function can be easily defined separately on every $g$-orbit). The function  $f$ extends to a continuous function on $\beta Z$ that we still denote by $f$. It follows that for every $\omega \in \cl{\Mov_g(Z)}$, we have $|f(g\omega)-f(\omega)| \geq 1$, and therefore $\omega \in \Mov_g(\beta Z)$, showing that $\cl{\Mov_g(Z)} = \Mov_g(\beta Z)$. This  implies in particular that the set $ \Mov_g(\beta Z)$ is clopen   
 for every $g\in G$, which is equivalent to the continuity of the stabilizer map (see, e.g., \cite[Lemma 2.2]{LeBoudec2016p}).
\end{proof}

Given a collection of subgroups $A \sub \Sub(G)$, we write $Z_A=\sqcup_{H\in A} G/H$ and endow it with the discrete topology. There is an obvious action $G \actson Z_A$, by letting $G$ act separately on each coset space.
\begin{prop}
  \label{p:proof-countable}
Let $G$ be a discrete group and $\cH \sub \Sub(G)$ be a closed invariant subset. Let $A \sub \cH$ be a subset such that the set of all conjugates of subgroups in $A$ is dense in $\cH$. Then the compact $G$-space $X = \beta Z_A$ verifies the conclusion of Theorem~\ref{t:Chabauty}.
\end{prop}

\begin{remark}
Of course, one can choose $A = \cH$. However, if $\cH$ is assumed to be a URS, then one can simply choose $A=\{H\}$ for any $H \in \cH$, so that $X = \beta (G/H)$.
\end{remark}

\begin{proof}
Continuity of the stabilizer map was already proved in Lemma~\ref{l:Stone-Cech:discrete}. Moreover, the image of $Z_A\sub \beta Z_A$ is a dense subset of $\cH$ by the assumption on $A$. Since $Z_A$ is dense in $\beta Z_A$, it follows that the image of $\beta Z_A$ is precisely $\cH$.
\end{proof}

For the reduction to a metrizable space when $G$ is countable discrete, we refer directly to the general case of a second countable locally compact group (cf. Proposition~\ref{p:metrizable}). However, we note that in this case, one can always choose a metrizable realization of the URS that is zero-dimensional.


\subsection{Case of locally compact groups}

Let $G$ be a locally compact group. We will always see $G$ as a uniform space endowed with the \df{right uniformity} whose entourages are
\begin{equation*}
  \cU_V = \set{(g_1, g_2) : \exists v \in V \ vg_1 = g_2},
\end{equation*}
where $V$ varies over symmetric neighborhoods of $1_G$. (Note that some authors call this the \emph{left} uniformity instead.)

A pseudometric $d$ on $G$ is called \df{right-invariant} if $d(g_1h, g_2h) = d(g_1, g_2)$ for all $g_1, g_2, h \in G$, and is said to be 
(right) uniformly continuous if it is uniformly continuous as a function $d\colon G\times G\to \R$. Note that every right-invariant, continuous pseudo-metric is uniformly continuous (see the argument in the proof of the next lemma). In the sequel, we will need the existence of uniformly continuous pseudometrics with some suitable properties.
\begin{lemma}
  \label{l:pseudometric}
  Let $g \in G$, $g \neq 1_G$ and let $U$ be a neighborhood of $1_G$. Then there exists a right-invariant, continuous pseudometric $d$ on $G$ such that:
  \begin{enumerate}
  \item $d \leq 8$;
  \item $1/2 \leq d(1_G, g) \leq 1$;
  \item the $d$-ball of radius $4$ around $1_G$ is relatively compact;
  \item \label{l:pseudometric:i:moves} $\set{x \in G : d(1_G, x) < 1/2} \sub U$.
  \end{enumerate}
\end{lemma}
\begin{proof}
  We adapt the argument of the proof of the Birkhoff--Kakutani metrization theorem from \cite{Berberian1974}. Without loss of generality, we may assume that $U$ is \df{symmetric} ($U = U^{-1}$), relatively compact, and that $g \notin U$. Let $V_0 = U \cup gU \cup (gU)^{-1}$, $V_{-1} = V_0^3$, $V_{-2} = V_{-1}^3$, $V_{-3} = G$; let $V_1$ be a symmetric neighborhood of $1_G$ such that $V_1^3 \subseteq U$, and for each $n \geq 1$, let $V_{n+1}$ be a symmetric neighborhood of $1_G$ such that $V_{n+1}^3 \subseteq V_n$. Thus for all $n \geq -3$, $V_n$ is symmetric and $V_{n+1}^3 \subseteq V_n$. Define $\rho \colon G^2 \to \R$ by
  \begin{equation*}
    \rho(x, y) = \inf \set{2^{-n} : xy^{-1} \in V_n}
  \end{equation*}
  and $d \colon G^2 \to G$ by
  \begin{equation*}
    d(x, y) = \inf \set{\sum_{i=0}^{k-1} \rho(x_i, x_{i+1}) : x_0 = x, x_k = y, x_1, \ldots, x_{k-1} \in G}.
  \end{equation*}
  We have that $\rho$ is symmetric, right-invariant and
  \begin{equation*}
    \rho(x_0, x_1) \leq \eps \And \rho(x_1, x_2) \leq \eps \And \rho(x_2, x_3) \leq \eps
      \implies \rho(x_0, x_3) \leq 2\eps.
  \end{equation*}

  By \cite{Berberian1974}*{Lemma~6.2}, $d$ is a right-invariant pseudometric on $G$ that satisfies
  \begin{equation*}
    \frac12 \rho(x, y) \leq d(x, y) \leq \rho(x, y), \quad \text{ for all } x, y \in G.
  \end{equation*}
By the triangle inequality and right invariance, we have
  \begin{equation*}
  |d(ux,vy)-d(x,y)|\leq d(ux,x)+d(vy,y)\leq\rho(u, 1_G)+\rho(v, 1_G),
  \end{equation*}
showing that $d$ is right uniformly continuous. Observe that $\rho$ may not separate points and that is why we obtain a pseudometric rather than a metric.

  We note that as $g \in V_0 \setminus V_1$, we have that $\rho(1_G, g) = 1$ and thus $1/2 \leq d(1_G, g) \leq 1$. Moreover, $\set{x \in G : d(1_G, x) < 4} \sub V_{-2}$ is relatively compact. Finally, if $x\notin V_1$, we have $\rho(1_G, x)\geq 1$ and thus $d(1_G, x)\geq 1/2$, proving that $\set{x \in G : d(1_G, x) < 1/2} \sub V_1\sub U$.
\end{proof}
\begin{remark}
  \label{rem:Struble}
  If $G$ is second countable, then by a result of Struble~\cite{Struble1974}, there always exists a proper, right-invariant metric on $G$ and in that case, one can use this metric instead of the pseudometric provided by Lemma~\ref{l:pseudometric} in what follows (with small modifications of the proof).
\end{remark}

Given a closed subgroup $H \leq G$, we equip the homogeneous space $G/H$ with  the quotient of the right uniformity of $G$. Explicitly, its  entourages are
\begin{equation*}
  \cU_V = \set{(g_1H, g_2H) : \exists v \in V \ vg_1H = g_2H},
\end{equation*}
where $V$ varies over symmetric neighborhoods of $1_G$.  If $d$ is a right-invariant,  continuous pseudometric on $G$, define $d_H$ on $G/H$ by
\begin{equation}
  \label{eq:dH}
  d_H(g_1H, g_2H) = \inf_{h \in H} d(g_1h, g_2).
\end{equation}
Note that by right invariance, $d_H$ is a pseudometric on $G/H$. Moreover, for every $g\in G$, we have $d_H(gx H, xH) \leq d(g, 1_G)$ which implies that $d_H$ is uniformly continuous.

Given $g\in G$ and $V \ni 1_G$, we denote
\begin{equation*}
  \Mov_g^V(G/H) = \set{xH \in G/H : gxH \notin VxH}.
\end{equation*}

The idea of the proof of the following lemma is adapted from the proof of Veech's theorem by Kechris, Pestov, and Todor\v{c}evi\'{c} \cite{Kechris2005}*{Appendix~A}.
\begin{lemma} \label{l:function}
 Let $g\in G$ and $V \ni 1_G$ be open. Let $H\le G$ be a closed subgroup. Then there exists $n\in \N$ and a uniformly continuous function $F \colon G/H \to \R^n$ with $\nm{F}_\infty \leq 8$ such that
\begin{equation*}
  \nm{F(gxH) - F(xH)}_\infty \geq 1/4 \quad \text{ for every } xH \in \Mov_g^V(G/H). 
\end{equation*}
Moreover, the dimension $n$ of the target $\R^n$ can be chosen to depend only on $g$ and $V$ but not on $H$.
\end{lemma}
\begin{proof}
  Choose a pseudometric $d$ as in Lemma~\ref{l:pseudometric} (with $U = V$). Define $d_H$ as in \eqref{eq:dH}. Using Zorn's lemma, choose a subset $A \sub G/H$ which is maximal with the property
\begin{equation*}
  aH, bH \in A \And aH \neq bH \implies d_H(aH, bH) \geq 1/8.
\end{equation*}
Define a graph $\Gamma$ with vertex set $A$ where $aH$ and $bH$ are connected by an edge if and only if $d_H(aH, bH) < 3$.

We claim that $\Gamma$ has bounded degree and that the bound on the degree does not depend on $H$. To see this, let $b_1H, \ldots, b_nH$ be distinct neighbors of $aH$. This means that there exist $h_1,\ldots, h_n \in H$ such that $d(a, b_ih_i) < 3$ for every $i = 1, \ldots, n$. Furthermore, by the definition of $A$, we have $d(b_ih_i, b_jh_j) \geq 1/8$ for every $i \neq j$. Since $d$ is right-invariant, this implies that the elements $x_i = b_ih_ia^{-1}$ lie in the ball of radius $3$ around $1_G$ and have distance at least $1/8$ between each other. It follows that their cardinality does not exceed the size $\ell$ of a finite cover by balls of radius $1/16$ of the ball of radius $3$ (which is relatively compact by Lemma~\ref{l:pseudometric}). Therefore $\Gamma$ has degree bounded by $\ell$.

It follows that $\Gamma$ can be colored with at most $n = \ell+1$ colors in such a way that no two adjacent vertices have the same color. Let $A = A_1 \sqcup \cdots \sqcup A_n$ be the resulting partition of the vertices. For every $i = 1, \ldots, n$, let $f_i \colon G/H \to \R$ be given by $f_i(xH) = d_H(xH, A_i)$. Set $F = (f_1, \ldots, f_n)$. 

Consider $xH \in \Mov_g^V(G/H)$ and note that this condition together with \ref{l:pseudometric:i:moves} in Lemma~\ref{l:pseudometric} implies that $d_H(xH, gxH) \geq 1/2$. By the definition of $A$, there exists a point $aH \in A$ such that $d_H(xH, aH) < 1/8$. Let $i$ be such that $aH\in A_i$. Then $f_i(xH) < 1/8$.

Next we examine $f_i(gxH)$. Observe that
\begin{equation*}
  \begin{split}
    d_H(gxH, aH) &\leq d_H(gxH, xH) + d_H(xH, aH) \\
      &\leq d(gx, x) + 1/8 \leq 9/8.
  \end{split}
\end{equation*}
We claim that $aH$ is the closest point in $A_i$ to $gxH$. Indeed, if another point in $A_i$ were closer to $xH$, it would have to lie at a distance less than $18/8$ from $aH$, which is not possible because two points in $A_i$ lie at distance at least $3$. Therefore
\begin{equation*}
  \begin{split}
    f_i(gxH) &= d_H(gxH, aH) \\
      &\geq d_H(gxH, xH) - d_H(xH, aH) \\
      &\geq 1/2 - 1/8 = 3/8.     
  \end{split}
\end{equation*}
We deduce that
\begin{equation*}
  \nm{F(gxH)- F(xH)}_\infty \geq |f_i(gxH) - f_i(xH)| \geq 3/8 - 1/8 = 1/4.
  \qedhere
\end{equation*}
\end{proof}

We are now ready to prove Theorem~\ref{t:Chabauty}. Let $\cH \sub \Sub(G)$ be a closed invariant subset. Let $A \sub \cH$ be such that the union of the orbits of elements of $A$ is dense in $\cH$. Let $Z = \sqcup_{H\in A} G/H$, endowed with the disjoint union topology and uniform structure.  For $g\in G$ and $V\ni 1_G$ open, we denote
\begin{equation*}
  \Mov_g^V(Z)=\{z\in Z : gz\notin Vz\} = \sqcup_{H\in A} \Mov_g^V(G/H).
\end{equation*}
As a consequence of the last sentence in Lemma~\ref{l:function} (stating that the dimension $n$ of the codomain of $F$ is uniform in $H$), if one is given $g$ and $V$, the functions $F$ obtained in Lemma~\ref{l:function} can be coalesced together to obtain a uniformly continuous function on $Z$, and therefore Lemma~\ref{l:function} remains valid for the uniform space $Z$. We record this in the next lemma. 
\begin{lemma}
  \label{l:functionZ}
  Let $g \in G$ and $V \ni 1_G$ be open. Then there exists $n \in \N$ and a bounded, uniformly continuous function $F \colon Z \to \R^n$ such that
  \begin{equation*}
    \nm{F(gz)-F(z)}_\infty \geq 1/4  \quad \text{ for every } z \in \Mov_g^V(Z).
  \end{equation*} 
\end{lemma}

Let $\Cub(Z)$ be the commutative $C^*$-algebra of bounded, uniformly continuous functions on $Z$ and let $S(Z)$ be its Gelfand spectrum (this is often called the \df{Samuel compactification of the uniform space $Z$}).
\begin{prop}
The $G$-space $X=S(Z)$ verifies the conclusion of Theorem~\ref{t:Chabauty}.
\end{prop}
\begin{proof}
  Since $Z$ is dense in $X$, it is enough to prove that for every $\omega \in X$ and every net $(z_i)_i \sub Z$ converging to $\omega$, the stabilizers $G_{z_i}$ converge to $G_\omega$. Fix $\omega \in X$ and a net $(z_i) \sub Z$ with $z_i \to \omega$. Let $K$ be a cluster point of $G_{z_i}$ and let us show that $K=G_\omega$.We may assume that $G_{z_i}$ converges to $K$. We have $K \leq G_\omega$ by upper semicontinuity of the stabiliser map.

  Towards a contradiction, suppose that the inclusion is strict and let $g \in G_\omega \setminus K$. Let $V \ni 1_G$ be a compact, symmetric neighborhood of $1_G$ such that $Vg \cap K = \emptyset$. This condition defines an open neighbourhood of $K$ in the Chabauty topology, so $g V \cap G_{z_i}=\varnothing$ for $i$ large enough. Equivalently (using that $V$ is symmetric) $z_i \in \Mov_g^V(Z)$. By Lemma~\ref{l:functionZ}, we can find a function $F \colon Z \to \R^n$ with the property that $\nm{F(gz_i) - F(z_i)}_\infty \geq 1/4$ for all $i$ large enough. Since $F$ extends to $X$, passing to the limit, we get $\nm{F(g \omega) - F(\omega)}_\infty \geq 1/4$. In particular, $g\omega \neq \omega$, contradicting the fact that $g \in G_\omega$. Therefore $K = G_\omega$ and the stabilizer map is continuous as claimed.

That the image of $\stab$ is equal to $\cH$ now follows from the fact that $\stab(Z)$ is a dense subset of $\cH$. 
\end{proof}

It remains to prove the claim of the last sentence in the statement of Theorem~\ref{t:Chabauty}.
\begin{prop}
  \label{p:metrizable}
 Let $X$ be the $G$-space constructed above. If $G$ is second countable, then there exists a metrizable quotient $Y$ of $X$ such that $\stab$ is continuous on $Y$ and $\stab(Y) = \cH$.
\end{prop}
\begin{proof}
Fix a countable basis $\cB$ at $1_G$. Let $Z = \sqcup_{H \in A} G/H$ as before. We will define the quotient $Y$ as the Gelfand space of a separable $G$-invariant subalgebra $\cA$ of $\Cub(Z)$. Note that for the $\stab$ map to be continuous on $Y$, we only need that the conclusion of Lemma~\ref{l:functionZ} holds for $\cA$, i.e.,
  \begin{equation*}
    \forall V \in \cB \ \forall g \in G \ \exists F \in \cA \quad \nm{F(gz) - F(z)}_\infty > 1/8 \quad \text{ for all } z \in \Mov_g^{V}(Z),
  \end{equation*}
  that is, the function $F = (f_1, \ldots, f_n) \colon Z \to \R^n $ can be chosen in such a way that $f_1, \ldots, f_n \in \cA$. Thus all we need to show is that for a fixed $V \in \cB$, there is a countable collection $\cA_V$ of functions $F \colon Z \to \R^n$ such that
  \begin{equation*}
    \forall g \in G \exists F \in \cA_V \ \forall z \in Z \quad gz \in Vz \quad\text{or}\quad \nm{F(gz) - F(z)}_\infty > 1/8.
  \end{equation*}
  Provided that this is done, we can take $\cA$ to be the smallest $G$-invariant, closed subalgebra that contains $\bigcup_{V \in \cB} \cA_V$, which is separable.
  
  Lemma~\ref{l:functionZ} and uniform continuity imply that for every $g \in G$, there exist $F \colon Z \to \R^n$ and an open $U \ni g$ such that 
  \begin{equation*}
    \forall g' \in U \ \forall z \in Z \quad g'z \in Vz \quad\text{or}\quad \nm{F(g'z) - F(z)}_\infty > 1/8.
  \end{equation*}
  Now the fact that $G$ is Lindelöf implies that we can find a countable collection of functions $F$ that works for all $g$.
\end{proof}


\section{Universal minimal flow relative to a URS}

\subsection{Existence and uniqueness}
If $\cH$ and $\cK$ are URSs of $G$, we write $\cH \preceq \cK$ if for all $H \in \cH$, there exists $K \in \cK$ such that  $H$ is contained in $K$. This relation is  a partial order on the set of URSs of $G$ (see \cite[Corollary 2.15]{LeBoudec2016p}; the proof given there for countable groups extends easily to locally compact groups), however we shall not use this fact.

\begin{defn}
  \label{df:domination}
  Let $G$ be a locally compact group, $G \actson X$ be a minimal action on a compact space $X$, and let $\cH$ be a URS of $G$. We will say that $X$ \df{is subordinate to} $\cH$ if $\cH \preceq \cS_G(X)$.
\end{defn}
Since $\cS_G(X)$ is in general  different from the collection of stabilizers of $G\actson X$, we make the following observation. 
\begin{lemma} \label{l:stab-contains-urs}
Let $G\actson X$ be a minimal action on a compact space which is subordinate to a URS $\cH$. Then every $H\in \cH$ fixes a point $x\in X$. 
\end{lemma}
\begin{proof}
We may assume $\cH=\cS_G(X)$. Since $\cS_G(X)$ is contained in the closure of point stabilizers, for every $H\in \cH$ there exists a net $x_i$ such that $G_{x_i}\to H$. By compactness we may assume that $x_i$ converges to some point $x$, and the conclusion follows from the upper semicontinuity of the stabilizer map.
\end{proof}



Recall that given two compact $G$-spaces $X$ and $Y$, we say that $X$ \emph{factors onto} $Y$ if there exists a continuous, surjective, $G$-equivariant map $X\to Y$. Given a collection $\mathcal{E}$ of compact $G$-spaces, we say that $X\in \mathcal{E}$ is \emph{universal} for $\mathcal{E}$ if it factors onto all elements of $\mathcal{E}$. 

The goal of this section is to establish the following theorem.
\begin{theorem}\label{t-UMF}
For every URS $\cH$ of $G$, there exists a minimal $G$-space $M(G, \cH)$, unique up to isomorphism, which is subordinate to $\cH$ and is universal for minimal $G$-spaces subordinate to $\cH$. Moreover, the stabilizer map $M(G, \cH) \to \cH$ is continuous.
\end{theorem}

\begin{defn}
The space $M(G, \cH)$ will be called the \df{universal minimal flow of $G$ relative to $\cH$}.
\end{defn}

If $\cH = \set[big]{\set{1_G}}$, then $M(G, \cH)$ is just the usual universal minimal flow of $G$.

The existence is easy. Let $H\in \cH$ be arbitrary and recall that $S(G/H)$ denotes the Samuel compactification of $G/H$. Let $M \sub S(G/H)$ be a minimal subset. Then $M$ verifies the universal property. Indeed, let $G \actson X$ be a minimal $G$ space subordinate to $\cH$. By Lemma \ref{l:stab-contains-urs} there exists a point $x \in X$ such that $H$ stabilizes $x$. The orbital map $G \to X, g \mapsto g \cdot x$ descends to a uniformly continuous map $G/H \to X$, which extends to a $G$-map $S(G/H)\to X$, and taking the restriction to $M$ shows that $M$ factors onto $X$. We have already proven that the stabilizer map is continuous and that the collection of stabilizers of $G\actson M$ is equal to $\cH$; in particular, $M$ is subordinate to $\cH$.

Our next goal is to check uniqueness. For this, it is enough to prove that $M$ is \df{coalescent}, i.e., that every continuous $G$-equivariant map $M\to M$  is a homeomorphism. For the usual (non-relative) universal minimal flow of $G$, this is a result of Ellis~\cite{Ellis1960}. Our proof is close to the exposition by Uspenskij~\cite{Uspenskij2000} of Ellis's theorem. In the classical case, the proof  is based on the fact that $S(G)$ carries a natural semigroup structure. The main difference is that in our case, $S(G/H)$ does not carry such a structure; however, we can find a semigroup inside $S(G/H)$ that is sufficient for our purposes.

Let $\Fix_H(M)$ be the set of points in $M$ fixed by $H$. Observe that for every $\omega \in \Fix_H(M)$, the orbital map $G/H \to G \cdot \omega$ extends to a continuous equivariant map $r_\omega \colon S(G/H) \to M$, which is moreover the unique $G$-map $S(G/H) \to S(G/H)$ sending $H$ to $\omega$. Hence, we get a map $S(G/H) \times \Fix_H(M)\to M$ continuous in the first variable. 

\begin{lemma} \label{l-semigroup}
For every $\omega \in \Fix_H(M)$, we have $r_\omega(\Fix_H(M)) \sub \Fix_H(M)$. In particular, $\Fix_H(M)$ is a right-topological semigroup under the operation $\Fix_H(M) \times \Fix_H(M) \to \Fix_H(M), (\eta, \omega) \mapsto \eta\omega:= r_\omega(\eta)$.
\end{lemma}
\begin{proof}
This is obvious because the map $r_\omega$ is $G$-equivariant.
\end{proof}

Since $\Fix_H(M)$ is a compact, right-topological semigroup, by a well-known theorem of Ellis, $\Fix_H(M)$ contains idempotent elements. 

\begin{lemma}
  \label{l:idempotent-retraction} \label{l-retraction-idempotent}
Let $\iota \in \Fix_H(M)$ be an idempotent. Then the map $r_\iota \colon S(G/H)\to M$ is a retraction of $S(G/H)$ onto $M$.
\end{lemma} 
\begin{proof}
We need to prove that $r_\iota|_M = \id$. Since $r_\iota(\iota) = \iota^2 = \iota$, by $G$-equivariance, $r_\iota$ is the identity on the orbit of $\iota$, which is dense in $M$ by minimality, whence the conclusion. 
\end{proof}

\begin{lemma}
  \label{l:romega}
Every continuous $G$-map $M \to M$ is of the form $r_\omega$ for some $\omega \in \Fix_H(M)$.
\end{lemma}
\begin{proof}
Let $f \colon M \to M$ be a continuous $G$-map. Let $\iota \in \Fix_H(M)$ be an idempotent. Consider $f \circ r_\iota \colon S(G/H)\to M$. As this map is continuous and equivariant, we have $f \circ r_\iota = r_\omega$ for $\omega = f(\iota)$. Since $r_\iota|_M = \id$, this implies that $f = r_\omega$. 
\end{proof}

\begin{prop} \label{p-coalescent}
$M$ is coalescent.
\end{prop}
\begin{proof}
  Let $f \colon M \to M$ be a continuous $G$-map. By minimality of $G\actson M$, $f$ is surjective. We need to show that $f$ is injective. By equivariance, we have $f(\Fix_H(M)) \sub \Fix_H(M)$. By Lemma~\ref{l:romega}, there exists $\omega \in \Fix_H(M)$ such that $f = r_\omega$. Therefore $f(\Fix_H(M)) = \Fix_H(M) \omega$ is a compact left ideal of $\Fix_H(M)$ and thus a compact subsemigroup of $\Fix_H(M)$. By Ellis's theorem, there exists an idempotent $\iota \in \Fix_H(M)\omega$. Let $\eta \in \Fix_H(M)$ be such that $f(\eta) = \eta\omega = \iota$. Let $g = r_\eta$. Now by Lemma~\ref{l:idempotent-retraction}, for all $x \in M$,
  \begin{equation*}
    (f \circ g)(x) = r_\omega (r_\eta(x)) = x\eta\omega = x\iota = r_\iota(x) = x.
  \end{equation*}
  The map $g$, being continuous and equivariant, is surjective by minimality. Since $f \circ g =\id$, it follows that $f$ is injective. 
\end{proof}
This concludes the proof of Theorem \ref{t-UMF}. It is worth pointing out the following corollary.
\begin{cor}
Let $H, H'\in \cH$ and let $M\subset S(G/ H)$ and $M'\subset S(G/H')$ be minimal $G$-invariant closed subsets. Then $M$ and $M'$ are homeomorphic as compact $G$-spaces.  
\end{cor}
\begin{proof}
We have shown that $M$ and $M'$ are both models for the universal space ${M}(G, \cH)$, and therefore they are isomorphic by Theorem \ref{t-UMF}. 
\end{proof}

\subsection{Minimal ideal structure}
We retain the notation from the previous subsection. The next proposition records some general properties of the semigroup $\Fix_H(M)$   (whose semigroup structure was defined in Lemma \ref{l-semigroup}). These properties are analogous to \cite[Proposition~I.2.3]{Gla-book} in the classical case. We are grateful to the anonymous referee for suggesting that they extend to this setting.
\begin{prop}
Let $G$ be a locally compact group, and $\cH$ be a URS of $G$. Let $H\in \cH$ and $M\subset S(G/H)$ be a closed minimal $G$-invariant subset. Consider the right topological semigroup $N=\Fix_H(M)$ and let $J\subset N$ be the set of idempotent elements of $N$. 
The following hold. 

\begin{enumerate}
\item \label{i-pglas-1} For every $\omega\in N$, we have $N\omega =N$, i.e., $N$ is $N$-minimal.
\item \label{i-pglas-2} For every $\eta\in J$ , the subset $\eta N$ is a group with unit element $\eta$.
\item \label{i-pglas-3} We have $N=\bigsqcup_{\eta\in J} \eta N$.
\item \label{i-pglas-4} All the groups $\eta N, \eta\in J$ are isomorphic to each other via the map $\eta N\to \theta N, \eta \omega \mapsto \theta \eta \omega =\theta \omega$ (for $\eta, \theta \in J$).
\item \label{i-pglas-5} For every $\eta\in J$, the map $\eta \omega\mapsto r_{\eta \omega}$ is an isomorphism of $\eta N$ onto $\Aut_G(M)$, the group  of $G$-automorphisms of $M$.
\end{enumerate}
\end{prop}
\begin{proof} For completeness, we repeat the arguments in \cite[Proposition~I.2.3]{Gla-book} with minor modifications.

\ref{i-pglas-1} Consider the map $r_\omega\colon M\to M$. It is clear that it is a continuous $G$-map, and it follows from Proposition \ref{p-coalescent} that it is invertible. By Lemma~\ref{l:romega}, there exists $\eta\in N$ such that $r_\eta=r_\omega^{-1}$. It follows that $N=r_\omega\circ r_\eta(N)= N\eta \omega\subset N\omega$. Conversely, it is clear that $N\omega \subset N$ and therefore $N=N\omega$. 

\ref{i-pglas-2} For every $\theta\in N$ we have $\eta (\eta \theta)= \eta^2\theta= \eta \theta$, showing that $\eta$ is a left unit in $\eta N$. The fact that it is a right unit follows from Lemma \ref{l-retraction-idempotent}. It remains to show that every element  $\omega\in\eta N$ has an inverse in $\eta N$. By part \ref{i-pglas-1}, we have $N\omega=N$ and therefore there exists $\theta \in N$ such that $\theta \omega=\eta$. Using part \ref{i-pglas-1} again, we have $N\theta=N$ and thus there exists $\rho \in N$ such that $\rho \theta=\eta$. It follows that $\omega =\eta \omega=(\rho \theta) \omega=\rho(\theta \omega)=\rho \eta= \rho$ (the last equality uses Lemma~\ref{l-retraction-idempotent}). Therefore $\eta \theta =(\theta \omega ) \theta =\theta (\rho \theta)= \theta \eta=\theta$. It follows that $\theta\in \eta N$ and $\theta \omega=\omega \theta=\eta$. 

\ref{i-pglas-3} The fact that the sets $\eta N, \eta\in J$ are pairwise disjoint is a consequence of \ref{i-pglas-2}. It remains to be checked that their union is equal to $N$. Let $\omega \in N$. The set $\{\theta \in N : \theta \omega= \omega\}$ is non-empty by part \ref{i-pglas-1} and therefore it is a non-trivial closed subsemigroup of $N$. By Ellis's theorem, it contains an idempotent $\eta$, and we have $\omega \in \eta N$. 

\ref{i-pglas-4} The claim that $\theta \eta \omega=\theta \omega$ follows from Lemma \ref{l-retraction-idempotent}. This shows in particular that the map in the statement is surjective. It is a group homomorphism because $\theta (\eta \omega)(\eta \omega')= \theta (\eta \omega \theta) (\eta \omega')=(\theta \eta \omega) (\theta \eta \omega')$ (the first equality uses \ref{i-pglas-2}).  Finally, it is invertible with inverse $\theta \omega\mapsto \eta \theta \omega=\eta \omega$ and therefore it is a group isomorphism. 

\ref{i-pglas-5} The map $\eta \omega \mapsto r_{\eta \omega}$ takes values in $\Aut_G(M)$ by Proposition \ref{p-coalescent}, and it is clear that it is a group homomorphism. If $r_{\eta \omega}=\operatorname{Id}_M$, then the element $\eta \omega$ is a right unit in $M$ which belongs to $\eta N$. Since the only right unit in $\eta N$ is $\eta$, we deduce that $\eta \omega=\eta$, showing that the map is injective.  To see that it is surjective, let $f\in \Aut_G(M)$. By Lemma \ref{l:romega}, there exists $\omega \in N$ such that $f=r_\omega$. But for every $\theta \in M$ we have $r_\omega(\theta)=\theta \omega=\theta \eta \omega =r_{\eta \omega}$. Therefore $r_{\eta \omega}=f$. \qedhere

\end{proof}

\subsection{Metrizability of $M(G, \cH)$}
It is a natural question for which pairs $(G, \cH)$ the relative universal minimal flow $M(G, \cH)$ can be identified with a more familiar, concrete $G$-space. A case in which this can be done is when the URS $\cH$ contains a cocompact subgroup $H$ (and thus is necessarily a single compact conjugacy class). In this case, $M(G, \cH)$ can be identified with the homogeneous space $G/H$. Our last result, whose proof relies on the results in \cite{BenYaacov2017}, says that there is little hope beyond this case.
\begin{theorem} \label{t:metrisable}
  Let $G$ be a locally compact second countable group and let $\cH$ be a URS of $G$. Then $M(G, \cH)$ is metrizable iff $\cH$ contains a cocompact subgroup.
\end{theorem}
\begin{proof}
  The ``if'' direction is clear. For the other, suppose that $M(G, \cH)$ is metrizable. Following the argument for the proof of Theorem~1.2 in \cite{BenYaacov2017}, we conclude that $M(G, \cH)$ contains a $G_\delta$ orbit $G \cdot x_0$. As $G$ is $\sigma$-compact, the orbit $G \cdot x_0$ is also $F_\sigma$, implying that its complement is $G_\delta$. If the complement is non-empty, it must be dense by minimality, contradicting the Baire category theorem. Thus the action $G \actson M(G, \cH)$ is transitive and if we put $H = G_{x_0}$, Effros's theorem (see, e.g., \cite{Hjorth2000}*{Theorem~7.12})  implies that $H$ is cocompact. As a consequence of Theorem \ref{t:Chabauty}, the point stabilizers of $G\actson M(G, \cH)$ are precisely the elements of $\cH$. Therefore $\cH$ contains a cocompact subgroup as claimed.
\end{proof}

\bibliography{realization-URS}
\end{document}